\documentclass[preprint]{imsart}

\usepackage{amsthm,amsmath,natbib}
\usepackage{amsfonts,amstext,amssymb, mathtools, comment}

\startlocaldefs
\newcommand{\levy}{L\'{e}vy }
\newcommand{\p}{{\mathbb P}}
\newcommand{\e}{{\mathbb E}}
\newcommand{\D}{{\mathrm d}}

\newcommand{\R}{{\mathbb R}}

\newcommand{\1}[1]{\mbox{\rm 1}_{\{#1\}}}

\renewcommand{\a}{{\alpha}}

\newcommand{\Cpos}{\mathbb C^{\Re> 0}}

\newtheorem{theorem}{Theorem}
\newtheorem{lemma}{Lemma}
\newtheorem{prop}{Proposition}
\newtheorem{remark}{Remark}

\endlocaldefs

\begin{document}
\bibliographystyle{imsart-nameyear}

\begin{frontmatter}

\title{Two coupled \levy queues with independent input}
\runtitle{Two coupled \levy queues}


\author{\fnms{Onno} \snm{Boxma}\ead[label=e1]{O.J.Boxma@tue.nl}}
\address{Department of Mathematics and Computer Science, Eindhoven University of Technology,\\5600 MB Eindhoven, The Netherlands\\ \printead{e1}}
\and
\author{\fnms{Jevgenijs} \snm{Ivanovs}\corref{}\ead[label=e2]{Jevgenijs.Ivanovs@unil.ch}\thanksref{t1}}
\thankstext{t1}{Supported by the Swiss National Science Foundation Project 200020-143889.}
\address{Department of Actuarial Science, University of Lausanne,\\CH-1015 Lausanne, Switzerland\\ \printead{e2}}


\begin{abstract}
We consider a pair of coupled queues driven by independent spectrally-positive \levy processes. 
With respect to the bi-variate workload process this framework includes both the coupled processor model and
the two-server fluid network with independent \levy inputs.
We identify the joint transform of the stationary workload distribution in terms of Wiener-Hopf factors corresponding to two auxiliary \levy processes
with explicit Laplace exponents. We reinterpret and extend the ideas of \cite{cohen_boxma} to provide 
a general and uniform result with a neat transform expression.
\end{abstract}




\end{frontmatter}

\bibliographystyle{acmtrans-ims}

\section{Introduction}
In the queueing literature, several studies have been devoted to a queueing model of two servers,
each with their own customer arrival process, with the special feature that the speed of one server
changes when the other server becomes idle. This has become known as the coupled processor model.
A possibly even more popular model of two servers is a fluid network with independent arrival processes, 
where fixed fractions of fluid exiting one queue are routed into the same and the other queue, as well as out of the system.
These models are intimately related and in the case of \levy input both can be put in our framework below.

More specifically, we assume that our queues are driven by two independent \levy processes $X_1(t)$ and $X_2(t)$ without negative jumps.
We model a pair of workload processes $(W_1(t),W_2(t))$ as a 2-dimensional reflected process, see e.g.~\cite{harrison_reiman_reflection2,kella_reflecting},
\begin{align}\label{eq:model}
W_1(t)=W_1(0)+X_1(t) -r_1 L_2(t)+L_1(t), \\
W_2(t)=W_2(0)+X_2(t) -r_2 L_1(t)+L_2(t),\nonumber
\end{align}
where $W_i(t)$ are nonnegative, $L_i(t)$ are nonnegative and nondecreasing with $L_i(0)=0$, and, in addition, it is required that if $t$ is a point of increase of $L_i(t)$
then $W_i(t)=0$. Sometimes the latter condition is replaced by an equivalent integral condition or minimality requirement. 
We assume that $r_1,r_2\geq 0$ and $r_1r_2<1$, in which case workload processes $(W_1(t),W_2(t))$ (with given initial values) and unused capacity processes $(L_1(t),L_2(t))$ satisfying the above conditions exist and are unique,
see~\cite[Sec.~5]{kella_reflecting}.

\subsection{Interpretations}
It is the easiest to understand the model given by~\eqref{eq:model} in the case of compound Poisson inputs and constant service rates \mbox{$c_i>0$}, i.e.\ when each $X_i(t)$ is a compound Poisson process (CPP) minus $c_it$.
Note that when $W_i(t)$ hits zero it stays at zero until the arrival of the next customer, which leads to the following four cases.
While \mbox{$W_1(t),W_2(t)>0$} these workload processes evolve according to $X_1(t)$ and $X_2(t)$. While $W_1(t)=0$ and $W_2(t)>0$ the process $L_1(t)$ evolves as $c_1t$
resulting in an additional service rate $r_2c_1$ in the second queue, i.e.\ the fraction $r_2$ of the first server capacity is used to help the second. Similarly, while $W_1(t)>0,W_2(t)=0$ the service rate in the first queue
is $c_1+r_1c_2$.
Finally, when both queues are empty, the processes $L_i(t)$ evolve as certain linear drifts canceling the negative drifts of $X_i(t)$ and each other's influence, which is possible if and only if $r_1r_2<1$. 
It is noted that compound Poisson input allows for a formulation of the coupled processor model, which goes beyond our assumption of $r_1r_2<1$.
One simply replaces~\eqref{eq:model} by an explicit description of the workload processes in the above four cases, see~\cite{cohen_boxma}.

As mentioned above, our model includes two-dimensional fluid networks with independent \levy input, where the column vector of workloads is a reflected process of the form:
\begin{align*}
W(t)=W(0)+\tilde X(t)-(I-P')ct+(I-P')\tilde L(t),
\end{align*}
see e.g.~\cite{kella_fluid_networks_levy}.
Here $\tilde X(t)$ is a column vector of external non-decreasing input processes into each queue, $P$ is a routing matrix (a substochastic matrix) with $P^n\rightarrow 0$ for $n\rightarrow \infty$, 
$P'$ is its transpose, $I$ is the identity matrix, and $c$ is a column vector of (maximal) service rates. One usually interprets $ct-\tilde L(t)$ as a vector of cumulative outflows from the queues, which are routed according to~$P$.
We can write
\[W_i(t)= W_i(0) + X_i(t)+(1-p_{ii})\tilde L_i(t)-p_{ji}\tilde L_j(t),\]
where $X_i(t)=\tilde X_i(t)-c_i t+p_{ii}c_i t+p_{ji}c_j t$ and $(i,j)$ is $(1,2)$ or $(2,1)$.
Letting $L_i(t)=(1-p_{ii})\tilde L_i(t)$ and $r_i=\frac{p_{ji}}{1-p_{jj}}$ we obtain~\eqref{eq:model} and guarantee the above conditions. 
We remark that commonly $\tilde X_i(t)$ is a subordinator (a non-decreasing \levy process) and hence $X_i(t)$ is a \levy process without negative jumps having bounded variation (on finite intervals). 
We allow $X_i(t)$ to be general \levy processes without negative jumps, which may lead to a certain debate about an appropriate model for the fluid network,
because cumulative outflows (if defined at all) are not necessarily non-decreasing in this general setup. Nevertheless, such models have appeared in the literature, see~\cite{kella_whitt_stability}. 
Finally, one can go the other way around and produce a network from the model~\eqref{eq:model}, which is immediate if $r_1,r_2\leq 1$. 
If $r_1>1$ (and similarly for $r_2>1$) then consider $(W_1(t),r_1 W_2(t))$ and note that it corresponds to  
a pair of workload processes in a network with routing matrix given by $p_{11}=p_{22}=0,p_{12}=r_1r_2,p_{21}=1$ and driving processes $X_1(t),r_1 X_2(t)$,
see also~\cite[Lem.~4.1]{kella_fluid_networks_levy}.

\subsection{Stationary distribution}
Let us note that $(I-P')^{-1}\e X(1)<0$, where $X(t)$ is a multidimensional driving process, is a sufficient condition 
for the existence of a stationary distribution in a general \levy network, which 
follows from~\cite{kella_whitt_stability}. Furthermore, if none of $X_i(t)$ is a zero process then this condition is also necessary.
Stronger limiting results are available in~\cite{kella_fluid_networks_levy} for the case when both $X_i(t)$ have bounded variation. 
Stability of~\eqref{eq:model} can be easily related to the stability of the corresponding network yielding the following condition
\begin{align}\label{eq:stability}
 &d_1+r_1d_2>0, & d_2+r_2d_1>0,
\end{align}
where $d_i=-\e X_i(1)$. Assuming that~\eqref{eq:stability} holds we let a pair of random variables $(W_1,W_2)$ refer to the joint stationary distribution of $(W_1(t),W_2(t))$ 
(uniqueness of this distribution will follow from the uniqueness of its transform).
Our main result is an expression for the Laplace-Stieltjes transform $\e e^{-\alpha_1 W_1 - \alpha_2 W_2}$
in terms of Wiener-Hopf factors corresponding to two auxiliary processes
with explicit Laplace exponents, see Theorem~\ref{thm:main}. We reinterpret and extend the ideas from Chapter III.3
of~\cite{cohen_boxma} to provide 
a general and uniform result. Its derivation is rather compact, and is based on a number of identities and observations from the fluctuation theory for \levy processes.

Let us shortly discuss a special case, when $X_1(t)$ is a subordinator, i.e.\ a non-decreasing \levy process. 
In this case $L_1(t)$ can increase only when $L_2(t)$ increases, hence both queues should be empty.
This feature allows for a rather simple analysis of the joint transform similarly to~\cite{kella_whitt_tandem}.
So we can assume in the following that each $X_i(t)$ is a spectrally-positive \levy process, i.e. it is a \levy process which is not a subordinator, and which can have only positive jumps.

\subsection{Related literature and motivation}
The main application/motivation of the coupled processor model is provided by the fact that, in a network of work stations, a user may use
other machines than its own when those machines are idle;
this is often referred to as {\em cycle stealing}.
Another application occurs in integrated-service communication networks.
Differentiated quality-of-service among different traffic flows is achieved in such networks
via scheduling algorithms such as Weighted Fair Queueing.
Mathematically, such scheduling algorithms may often be represented by a form of {\em Generalized Processor Sharing},
where traffic flow $i$ gets a weight factor $w_i \in (0,1)$, with $\sum w_i =1$.
If all traffic flows are backlogged, then flow $i$ is served at rate $w_i$.
If some of the flows are not backlogged, then the excess capacity is redistributed
among the backlogged flows proportionally to their weights.
Again, this may be viewed as a form of cycle stealing.
A pioneering paper on the mathematical analysis of coupled processors is~\cite{FI}.
They consider two $M/M/1$ queues with service speeds $c_1$ and~$c_2$, respectively, unless the other queue is empty;
then the speeds are $c_1^*$ and~$c_2^*$, respectively.
They study the two-dimensional queue length process, and show how the generating function of the joint steady-state
queue length distribution can be obtained via the solution of a Riemann-Hilbert boundary value problem.
\cite{KMM} provide an elegant solution of the special, slightly easier, case of two symmetric $M/M/1$ queues in which a server doubles its speed
when the other server is idle (one might say that the idle server helps the other one).
\cite{cohen_boxma} have generalized the model of \cite{FI} to the case of general service time distributions.
They consider the two-dimensional workload process. See~\cite{Cohen92} for a further extension
to the case that arrivals may also simultaneously occur at both queues.

The analytic approach of \cite{FI,cohen_boxma}, exploiting a relation to boundary value problems in complex function theory,
seems to be limited to two dimensions. This has led to work in the following directions.\hfill\\
(i) Application of a numerical-analytic method, the Power Series Algorithm, gives numerical results
for more than two coupled processors \cite{Blanc,HKR}.\hfill\\
(ii) \cite{OHS} have developed an approximation method which yields  mean response times; the approximation can be made as accurate
as desired.\hfill\\
(iii) Several studies (see, e.g., \cite {BBJ2,BBJ1}) consider tail asymptotics of workloads for coupled processors
and multi-queue systems with some form of Generalized Processor Sharing.

The body of literature concerning fluid networks with \levy input is huge. 
The joint transform of the stationary workload in such networks is not known apart from a few special cases.
The transform can be obtained for tandem and feed-forward networks with decreasing service rates (in the direction of flow), see e.g.~\cite{kella_whitt_tandem}.
In such networks, if one queue is empty then all the queues preceding it are empty as well. This is the main feature facilitating the computation,
which can be also guaranteed in some other models, see~\cite{badila_boxma}.
Otherwise, the only tractable examples concern networks of two queues, which are closely related to the coupled processor model discussed above. 
The main result of the present paper yields an exact expression for the
transform in a two-node network with independent \levy input. This model generalizes, e.g.,~a tandem queue of~\cite{MR}, for which the authors study tail asymptotics.

\subsection{Organization of the paper}
Section~\ref{sec:facts} summarizes basic facts about spectrally-positive \levy processes and about Wiener-Hopf factorization
of \levy processes.
In Section~\ref{sec:funda} we relate a spectrally-positive \levy process to a certain pure-jump subordinator
which plays a fundamental role in our main result, Theorem~\ref{thm:main}, formulated in Section~\ref{sec:main}.
Section~\ref{sec:proof} contains the proof of Theorem~\ref{thm:main}.
It basically consists of three steps. In Subsection~\ref{sec:proof1} we derive a functional equation
for the joint workload transform; in Subsection~\ref{sec:proof2} the kernel of that functional equation is studied,
and the functional equation is solved via Wiener-Hopf factorization assuming certain bounds;
these bounds are established in Subsection~\ref{sec:proof3}.
Some special cases are considered in Section~\ref{sec:special}, where we also discuss the result of~\cite{cohen_boxma} in the case of CPP inputs.

\section{Basic facts}\label{sec:facts}
For ease of reference let us recall the L\'evy-Khintchine formula for a spectrally-positive \levy process $X(t)$
(cf.~\cite{kypr}):
\begin{equation}\label{eq:phi}
 \phi(\a)=\log\e e^{-\a X(1)}=a\a+\frac{1}{2}\sigma^2\a^2-\int_0^\infty(1-e^{-\a x}-\a x\1{x<1})\nu(\D x),
\end{equation}
where $\nu(\D x)$ is a
\levy measure on $(0,\infty)$ satisfying $\int_0^\infty(1\wedge
x^2)\nu(\D x)<\infty$. 
The process $X(t)$ has bounded variation (on finite intervals) if and only if $\sigma=0$ and $\int_0^1x\nu(\D x)<\infty$, in which case we have an alternative representation
\begin{equation}\label{eq:phi_bounded}
 \phi(\a)=\mu\a-\int_0^\infty(1-e^{-\a x})\nu(\D x)
\end{equation}
and $\mu$ can be interpreted as a linear drift. The case of $\mu=0$ corresponds to a pure-jump subordinator.
This subordinator is either a compound Poisson process (CPP) or an infinite activity subordinator according to $\nu(0,\infty)$ being finite or infinite.

Differentiating under the integral sign in~\eqref{eq:phi}, which can be justified, we get
\[\phi'(\a)=a+\sigma^2\a+\int_0^1x(1-e^{-\a x})\nu(\D x)-\int_1^\infty xe^{-\a x}\nu(\D x)\]
for $\a>0$. This shows that $X(t)$ has bounded variation if and only if $\lim_{\a\rightarrow\infty}\phi'(\a)$ is finite.

Finally, let us recall the celebrated Wiener-Hopf factorization for a general (two-sided) \levy process $X(t)$ and some positive constant $p>0$, see also~\cite[Thm.\ 6.16]{kypr} (and comments on p.~167 about the CPP case).
Consider the Laplace transforms
\begin{align}
 &\Psi^+(\a)=\e e^{-\a \overline X(e_p)}, ~~ \Re(\a)\geq 0, &\Psi^-(\a)=\e e^{-\a \underline X(e_p)}, ~~ \Re(\a)\leq 0,
\end{align}
where $e_p$ is an independent exponentially distributed r.v.\ with rate $p$ and $\overline X(t),\underline X(t)$ denote supremum and infimum processes respectively.
Note that $\Psi^\pm(\a)$ are analytic in the corresponding half-planes and continuous on the imaginary axis.
They satisfy the following factorization for $w\in i\R$:
\begin{equation}\label{eq:WH}
 \frac{p}{p-\phi(w)}=\Psi^+(w)\Psi^-(w).
\end{equation}
Let us finally note that identification of the Wiener-Hopf factors is a difficult but well-studied problem with some numerical evaluation techniques available, see e.g.~\cite{mandjes_gruntjes}.

\section{Fundamental subordinators}
\label{sec:funda}
Consider a spectrally-positive \levy process $X(t)$, which will serve as a driving process in our model.
The goal of this section is to associate to $X(t)$ a certain pure-jump subordinator $Y(t)$, which will play a fundamental role in our main result. 
Recall that $\phi(\a)$ denotes the Laplace exponent of $X(t)$, and $d=-\e X(1)=\phi'(0)\in(-\infty,\infty)$.
Note that we have excluded only $d=-\infty$, which is allowed because of the stability condition~\eqref{eq:stability}.

Consider the first passage (downwards) process $\tau_x^-,x\geq 0$, where $\tau_x^-=\inf\{t\geq 0:X(t)<-x\}$, 
which is a (possibly killed) \levy subordinator with the Laplace exponent $-\Phi(\a)$ defined via
 \begin{equation}\label{eq:Phi}
  \e e^{-\a\tau_x^-}=e^{-\Phi(\a)x}
 \end{equation}
for all $\a$ with $\Re(\a)\geq 0$. 
For real positive $\a$ the function $\Phi(\a)$ is positive and is uniquely identified by $\phi(\Phi(\a))=\a$. 
Moreover, $\lim_{\a\rightarrow\infty}\Phi(\a)=\infty$, and also $\Phi(0)=0$ if and only if $d\geq 0$.

It is known that $-\a/\Phi(\a)$ is the Laplace exponent of a certain killed subordinator (ascending ladder time process, see e.g.~\cite[p.~170]{kypr}). 
Note that if $d\geq 0$ then $\lim_{\a\downarrow 0}\frac{-\a}{\Phi(\a)}=\frac{-1}{\Phi'(0)}=-\phi'(0)=-d$. If, however, $d<0$ then this limit is 0.
Hence $d^+=d\vee 0$ is the rate of killing, which we remove to obtain a non-killed subordinator $Y(t)$ with the Laplace exponent
\begin{equation}\label{eq:phi_y}
\phi^Y(\a)=d^+-\frac{\a}{\Phi(\a)}.
\end{equation}
This is a pure-jump subordinator, which follows from $\lim_{\a\rightarrow \infty}\phi^Y(\a)/\a=0$ and representation~\eqref{eq:phi_bounded}.
Note also that~\eqref{eq:phi_y} holds for all $\a\neq 0$ with $\Re(\a)\geq 0$ by analyticity and continuity of Laplace exponents, see Section~\ref{sec:proof2}. 

Let us consider the dichotomy of bounded and unbounded variation for the process $X(t)$:
\begin{enumerate}
 \item $X(t)$ is of bounded variation: $Y(t)$ is a CPP,
\item $X(t)$ is of unbounded variation: $Y(t)$ is an infinite activity subordinator.
\end{enumerate}
This can be seen by considering $\lim_{\a\rightarrow\infty}\phi^Y(\a)=d^+-\lim_{\a\rightarrow\infty}\frac{1}{\Phi'(\a)}=d^+-\lim_{\a\rightarrow\infty}\phi'(\a)$,
which is finite in the first case and is $-\infty$ in the second as was discussed in Section~\ref{sec:facts}.
So in the first case we have $\p(Y(1)=0)>0$ and in the second $\p(Y(1)=0)=0$, which correspond to a CPP and an infinite activity subordinator respectively.

\section{Transform of the stationary workload}
\label{sec:main}
Consider the model specified by~\eqref{eq:model}, where $r_i\geq 0$ and $r_1r_2<1$. Recall that $X_1(t)$ and $X_2(t)$ are two independent spectrally-positive \levy processes with 
Laplace exponents $\phi_i(\a)$, $d_i=-\e X_i(1)\in(-\infty,\infty)$, and assume that stability condition~\eqref{eq:stability} holds.
Let $Y_i(t)$ be a pure-jump subordinator associated to $X_i(t)$, whose Laplace exponent $\phi^Y_i(\a)$ is given in~\eqref{eq:phi_y}.
Define two \levy processes and two positive constants:
\begin{align}\label{eq:X_LR}
&X_L(t)=Y_1(r_2t)-Y_2(t), &X_R(t)=Y_1(t)-Y_2(r_1t),\\
&p_L=d^+_2+r_2d^+_1, &p_R=d^+_1+r_1d^+_2.\nonumber
\end{align}
Their corresponding Laplace exponents for $w \in i\R, w\neq 0$ are given by
\begin{align}\label{eq:phi_phi}
 &\phi_L(w)=p_R-r_2\frac{w}{\Phi_1(w)}+\frac{w}{\Phi_2(-w)}, &\phi_R(w)=p_L-\frac{w}{\Phi_1(w)}+r_1\frac{w}{\Phi_2(-w)}.
\end{align}
Finally, we let $\Psi_L^\pm(\a)$ be the Wiener-Hopf factors corresponding to $X_L(t)$ and rate parameter $p_L$. 
Similarly, $\Psi_R^\pm(\a)$ are the Wiener-Hopf factors corresponding to $X_R(t)$ and $p_R$. 

\begin{theorem}\label{thm:main}
The joint transform of the stationary workloads is given by
\begin{align}\label{eq:thm}
&\e e^{-\a_1 W_1-\a_2 W_2}=\frac{1}{(1-r_1r_2)(\phi_1(\a_1)+\phi_2(\a_2))}\\
&\times\left(p_R^0(\a_1-r_2\a_2)\frac{\Psi_L^-(-\phi_2(\a_2))}{\Psi_R^-(-\phi_2(\a_2))}+p_L^0(\a_2-r_1\a_1)\frac{\Psi_R^+(\phi_1(\a_1))}{\Psi_L^+(\phi_1(\a_1))}\right),\nonumber
\end{align}
where $\a_1> \Phi_1(0),\a_2> \Phi_2(0)$, and
\begin{align}\label{eq:p0}
 &p_L^0=d_2+d_1^+r_2+d_1^-/r_1, &p_R^0=d_1+d_2^+r_1+d_2^-/r_2.
\end{align}
\end{theorem}
It is noted that if $d_1,d_2\geq 0$ then $p_L^0=p_L$ and $p_R^0=p_R$. Moreover, if $r_i=0$ then $d_i>0$ according to~\eqref{eq:stability} implying $d_i^-=0$. In this case $0/0$ in the definition of $p^0_R,p_L^0$ is interpreted as $0$.

Consider the above systems of queues for $r_1=r_2=0$. Then $\overline X_R(t)=Y_1(t)$ and $\underline X_R(t)=0$ for all $t$, and $p_R=d_1$.
From the definition of the W-H factors we have $\Psi_R^-(\a)=1$ and $\Psi_R^+(\a)=\e e^{-\a Y_1(e_{d_1})}=\frac{d_1}{d_1-\phi^Y_1(\a)}$.
Plugging in $\a=\phi_1(\a_1)$ we obtain $\Psi_R^+(\phi_1(\alpha_1)) = \frac{d_1\a_1}{\phi_1(\a_1)}$, and similarly we obtain expressions for the other terms.
Putting things together we see that~\eqref{eq:thm} becomes
\[
\frac{1}{(\phi_1(\a_1)+\phi_2(\a_2))}
\left(d_1\a_1\frac{d_2\a_2}{\phi_2(\a_2)}+d_2\a_2\frac{d_1\a_1}{\phi_1(\a_1)}\right)=\frac{d_1\a_1}{\phi_1(\a_1)}\frac{d_2\a_2}{\phi_2(\a_2)},\]
which is indeed the transform of the workload in two independent queues.
Another verification of Theorem~\ref{thm:main} is given in Section~\ref{sec:detdrift}, where we assume that $X_2(t)=-d_2t$ for $d_2>0$.
For such a (degenerate) system we first provide a quick alternative derivation of the joint transform and then check it against Theorem~\ref{thm:main}.

\section{Proof}
\label{sec:proof}
In this section we prove Theorem~\ref{thm:main}.
The proof consists of three steps. In Subsection~\ref{sec:proof1} we derive a functional equation
for the joint workload transform; in Subsection~\ref{sec:proof2} the kernel of that functional equation is studied,
and the functional equation is solved via Wiener-Hopf factorization assuming certain bounds;
these bounds are established in Subsection~\ref{sec:proof3}.

\subsection{The functional equation}
\label{sec:proof1}
In this section we derive an equation for the two-dimensional joint workload transform, which involves two unknown functions. 
Identification of these functions is the main problem, which will be addressed in Subsections~\ref{sec:proof2} and
\ref{sec:proof3}.
The following result is based on a by now standard argument using the Kella-Whitt martingale, see~\cite{kella_whitt}.  
\begin{prop}
 It holds that 
\begin{align}\label{eq:main}&(\phi_1(\a_1)+\phi_2(\a_2))\e e^{-\a_1 W_1-\a_2 W_2}=(\a_1-r_2\a_2)F_1(\a_2)+(\a_2-r_1\a_1)F_2(\a_1),
\end{align}
where $\a_1,\a_2\geq 0$ and 
\begin{align}\label{eq:F}
 &F_1(\a)=\e^* \int_0^1 e^{-\a W_2(t)}\D L_1(t), &F_2(\a)=\e^* \int_0^1 e^{-\a W_1(t)}\D L_2(t)
\end{align}
and $\e^*$ is the expectation in stationarity, i.e.\ we assume that $(W_1(0),W_2(0))$ is distributed as $(W_1,W_2)$.
\end{prop}
\begin{proof}
 Fix $\a_1,\a_2>0$ and define a spectrally-positive \levy process $X(t)=\a_1 X_1(t)+\a_2 X_2(t)$ 
and a process of bounded variation $Y(t)=(\a_1-r_2\a_2)L_1(t)+(\a_2-r_1\a_1)L_2(t)$, so that $Z(t):=\a_1 W_1(t)+\a_2 W_2(t)=X(t)+Y(t)$.
Let us first show that $\e L_i(t)<\infty$ and hence the expected variation of $Y(t)$ on finite intervals is finite. Start by noting that 
\begin{align*}
&L_1(t)\leq -\underline X_1(t)+r_1 L_2(t),& L_2(t)\leq -\underline X_2(t)+r_2 L_1(t),
\end{align*}
see also~\cite{kella_reflecting}. 
Hence $(1-r_1r_2)L_1(t)\leq -\underline X_1(t)-r_1 \underline X_2(t)$, but it is known that $\e |\underline X_i(t)|<\infty$. 
In conclusion,
\[M_t=(\phi_1(\a_1)+\phi_2(\a_2))\int_0^t e^{-Z(s)}\D s+e^{-Z(0)}-e^{-Z(t)}-\int_0^t e^{- Z(s)}\D Y(s)\]
is a martingale for any initial distribution, see~\cite[Thm.~2]{kella_whitt}.
Considering $\e^* M_1=0$ we obtain
\begin{align*}
 &(\phi_1(\a_1)+\phi_2(\a_2))\e e^{-\a_1 W_1-\a_2 W_2}=\\
&(\a_1-r_2\a_2)\e^*\int_0^1 e^{-Z(s)}\D L_1(s)+(\a_2-r_1\a_1)\e^*\int_0^1 e^{-Z(s)}\D L_2(s).
\end{align*}
Use the properties of $L_i(t)$ to conclude.
\end{proof}

\begin{remark}
 The functional equation~\eqref{eq:main} can be used to derive the following identity for the means:
\begin{equation}\label{eq:means}
 r_2(d_1+r_1 d_2)\e W_1+r_1(d_2+r_2 d_1)\e W_2=\frac{1}{2}(r_2\phi''_1(0)+r_1\phi''_2(0)).
\end{equation}
One way is to put $\a_1=r_2\a,\a_2=\a$, express $F_2(r_2\a)$, differentiate it at $\a=0$, and then to do the same for $\a_1=\a,\a_2=0$. Then the above identity follows by 
expressing $F_2'(0)$ from these equations. 
Note also that if $d_1,d_2>0$ then the right side of~\eqref{eq:means} is $r_2 d_1\e V_1+r_1 d_2\e V_2$, 
where $V_i$ refers to the stationary workload in queue $i$ considered alone.
It may be an interesting exercises to prove this relation
probabilistically from the first principles at least for Poisson
inputs.
\end{remark}

\subsection{Wiener-Hopf factorization}\label{sec:proof2}
Start by noting that the Laplace exponent $\phi_i(\a)$ of a spectrally-positive \levy process is analytic in the right half of the complex plane and is continuous on the imaginary axis,
which can be shown from~\eqref{eq:phi}. Hence the same is true for $\Phi_i(\a)$, which is minus a Laplace exponent according to~\eqref{eq:Phi}.
This equation also implies that $\Re(\Phi_i(\a))>0$ if and only if $\Re(\a)>0$, which is seen by sending $x$ to $\infty$, and that $\Phi_i(\a)\neq 0$ for $\a\neq 0$.
Hence the identity $\phi_i(\Phi_i(\a))=\a$ extends from $\a>0$ to all $\a$ with $\Re(\a)\geq 0$. 
Note also that the functions $F_i(\a)$ are analytic on $\{\a:\Re(\a)>0\}$ and continuous on the imaginary axis, which follows from their definition and $\e^* L_i(1)<\infty$.
In conclusion, Equation~\eqref{eq:main} holds for all $\a_1,\a_2$ with $\Re(\a_1),\Re(\a_2)\geq 0$. 

We use a similar uniformization approach
as in ~\cite{cohen_boxma} for the special
case of compound Poisson input: we consider~\eqref{eq:main} for $\a_1=\Phi_1(w)$ and $\a_2=\Phi_2(-w)$, where $w\in i\R$ lies on the imaginary axis, and obtain
\[(\Phi_1(w)-r_2\Phi_2(-w))F_1(\Phi_2(-w))+(\Phi_2(-w)-r_1\Phi_1(w))F_2(\Phi_1(w))=0.\]
Assuming $w\neq 0$ we multiply this equation by $\frac{w}{\Phi_1(w)\Phi_2(-w)}$ to get 
\begin{equation}\label{eq:tobe_factorized}
 (-\frac{w}{\Phi_2(-w)}+r_2\frac{w}{\Phi_1(w)})F_1(\Phi_2(-w))=(\frac{w}{\Phi_1(w)}-r_1\frac{w}{\Phi_2(-w)})F_2(\Phi_1(w)),
\end{equation}
which immediately translates into
\[(p_L-\phi_L(w))F_1(\Phi_2(-w))=(p_R-\phi_R(w))F_2(\Phi_1(w))\]
according to~\eqref{eq:phi_phi}.
Finally, from the Wiener-Hopf factorization~\eqref{eq:WH} we have
\begin{equation}\label{eq:imaginary}
 p_L\frac{\Psi_R^-(w)}{\Psi_L^-(w)}F_1(\Phi_2(-w))=p_R\frac{\Psi_L^+(w)}{\Psi_R^+(w)}F_2(\Phi_1(w)),
\end{equation}
which also holds for $w=0$ by continuity.

The left side of~\eqref{eq:imaginary} is analytic in the left half-plane and the right side is analytic in the right-half plane, and both are continuous and coincide on the boundary.
So one is an analytic continuation of the other, see~\cite[Thm IX.1.1]{lang_complex_analysis}.
Assume for a moment that the so-obtained entire function is bounded. Then by Liouville's theorem~\cite[Thm III.7.5]{lang_complex_analysis} it is a constant, call it $C$.

Let us determine the constant $C$, by plugging $w=0$ in~\eqref{eq:imaginary}. According to the stability condition~\eqref{eq:stability} at least one of $d_i$ is positive. 
If $d_1>0$ then $\Phi_1(0)=0$ and hence $C=p_R\e^* L_2(1)$, whereas $C=p_L\e^* L_1(1)$ if $d_2>0$.
Note also that for a stationary system
\[0=-d_1-r_1 \e^* L_2(1)+\e^* L_1(1)\text{ and }0=-d_2-r_2 \e^* L_1(1)+\e^* L_2(1),\]
which yields
\begin{align*}
 &\e^* L_1(1)=\frac{d_1+r_1d_2}{1-r_1r_2}, &\e^* L_2(1)=\frac{d_2+r_2d_1}{1-r_1r_2}
\end{align*}
and provides the expression for $C$. Furthermore,
\begin{align*}
& F_1(\Phi_2(-w))=\frac{p_R^0}{1-r_1r_2}\frac{\Psi_L^-(w)}{\Psi_R^-(w)}, &\Re(w)\leq 0,\\ &F_2(\Phi_1(w))=\frac{p_L^0}{1-r_1r_2}\frac{\Psi_R^+(w)}{\Psi_L^+(w)}, &\Re(w)\geq 0,
\end{align*}
where $p_L^0,p_R^0$ are given in~\eqref{eq:p0}. This can be checked by considering three scenarios $d_i\geq 0, d_2<0, d_1<0$ separately.

Considering the first equation we let $w=-\phi_2(\a)$ for $\a\geq \Phi_2(0)$, so that $\Phi_2(-w)=\a$; for the second we let $w=\phi_1(\a)$ with $\a\geq \Phi_1(0)$.
This immediately yields the functions $F_i(\a)$:
\begin{align*}
&F_1(\a)=\frac{p_R^0}{1-r_1r_2}\frac{\Psi_L^-(-\phi_2(\a))}{\Psi_R^-(-\phi_2(\a))}, & F_2(\a)=\frac{p_L^0}{1-r_1r_2}\frac{\Psi_R^+(\phi_1(\a))}{\Psi_L^+(\phi_1(\a))}.
\end{align*}
This together with~\eqref{eq:main} completes the proof of Theorem~\ref{thm:main} under the assumption 
that the entire function defined by~\eqref{eq:imaginary} is a constant.

\subsection{Bounds on the entire function}\label{sec:proof3}
In this section we show that the entire function defined by~\eqref{eq:imaginary} is a constant.
Consider~\eqref{eq:F} and observe that $F_1(\a)$ and $F_2(\a)$ are bounded for $\Re(\a)\geq 0$.
Let $\tilde X_L(t)$ and $\tilde X_R(t)$ be the processes $X_L(t)$ and $X_R(t)$, see~\eqref{eq:X_LR}, in the model
defined by \eqref{eq:model} but with interchanged indices. 
That is, we consider the same system with reversed indexing. Observe that $X_R(t)=-\tilde X_L(t),X_L(t)=-\tilde X_R(t),p_R=\tilde p_L, p_L=\tilde p_R$ and hence
\begin{align}\label{eq:interchanged}
&\Psi_R^-(-w)=\tilde \Psi_L^+(w), &\Psi_L^-(-w)=\tilde \Psi_R^+(w).  
\end{align}
Therefore, it is sufficient to analyze $\Psi_L^+(w)/\Psi_R^+(w),\Re(w)\geq 0$.

Recall the following Spitzer-type identity for a \levy process $X_L(t)$:
\begin{equation}\label{eq:Psi}
 \Psi_L^+(w)=\exp\left(-\int_0^\infty\int_0^\infty (e^{-p_Lt}-e^{-p_Lt-wx})\frac{1}{t}\p(X_L(t)\in \D x)\D t\right),
\end{equation} 
see Thm.\ 6.16 and comments on p.~168 in~\cite{kypr}.
Observe also that for $\Re(w)\geq 0$ we have
\begin{align*}
  &|\int_1^\infty\int_0^\infty (e^{-p_Lt}-e^{-p_Lt-wx})\frac{1}{t}\p(X_L(t)\in \D x)\D t|\leq \int_1^\infty 2e^{-p_Lt}\frac{1}{t}\D t<\infty, \\
 &|\int_0^1\int_0^\infty e^{-w x}(1-e^{-p_Lt})\frac{1}{t}\p(X_L(t)\in \D x)\D t| \leq p_L.
\end{align*}
Hence $\Psi_L^+(w)/\Psi_R^+(w)$ is bounded by $C_1|e^{A(w)}|$, where 
\begin{align}\label{eq:A}
&A(w):=\\ \int_0^1 \frac{1}{t}&\left(\int_0^\infty (1-e^{-wx})\p(X_R(t)\in \D x)-\int_0^\infty (1-e^{-wx})\p(X_L(t)\in \D x)\right)\D t. \nonumber
\end{align}

\subsubsection{Bounded variation case}
If both $X_1(t)$ and $X_2(t)$ have bounded variation then $X_L(t)$ and $X_R(t)$ are CPPs and hence 
\begin{align*}
&\int_0^1 \frac{1}{t}\p(X_L(t)>0) \D t<\infty, &\int_0^1 \frac{1}{t}\p(X_R(t)>0) \D t<\infty.
\end{align*}
This immediately shows that 
$A(w)$ and hence $\Psi_L^+(w)/\Psi_R^+(w)$ are bounded.

\subsubsection{The general case}
The proof in the general case is based on the following proposition.
\begin{prop}\label{prop:cond}
There exists a constant $C$, such that $|\Psi_L^+(w)/\Psi_R^+(w)|$, $\Re(w)\geq 0$ is bounded by $C|w|^4$ for large enough $|w|$.
In addition, \linebreak
$\Psi_L^+(w)/\Psi_R^+(w)=o(w)$ as $w\rightarrow\infty$ along the real numbers.
\end{prop}
Proposition~\ref{prop:cond} together with~\eqref{eq:interchanged} implies that the entire function defined by~\eqref{eq:imaginary} is bounded by a polynomial 
and hence it is a polynomial itself, see~\cite[Cor.\ III.7.4]{lang_complex_analysis}. Taking limit along the reals shows that this polynomial is just a constant.
The proof of Proposition~\ref{prop:cond} relies on the following technical lemma.
\begin{lemma}\label{lem:Gamma} For $\Re(z)\geq 0$ it holds that 
\[\int_0^1\frac{1}{t}|1-e^{-zt}|\D t\leq 2+2\ln(|z|\vee 1).\]
 For positive $z$ the bound can be replaced by $1+\ln(|z|\vee 1)$. 
\end{lemma}
\begin{proof}
For $|z|>1$ we write
\begin{align}\label{eq:lemma}
 &\int_0^1\frac{1}{t}|1-e^{-zt}|\D t=\int_0^{|z|}\frac{1}{t}|1-e^{-tz/{|z|}}|\D t\\
&\leq \int_0^1\frac{2|tz/|z||}{t}\D t+\int_1^{|z|}\frac{2}{t}\D t=2+2\ln |z|.\nonumber
\end{align}
The result is immediate for $|z|<1$. 
\end{proof}

\begin{proof}[Proof of Proposition~\ref{prop:cond}]
For positive $w$ the function $A(w)$ defined in~\eqref{eq:A} is bounded from above by
\[\int_0^1 \frac{1}{t}\e(1-e^{-w Y_1(t)};X_R(t)>0)\D t\leq\int_0^1 \frac{1}{t}(1-e^{\phi^Y_1(w)t})\D t, \]
because $X_R(t)\leq Y_1(t)$. Recall that $\phi^Y_1(w)\leq 0$, and use Lemma~\ref{lem:Gamma}, to bound $A(w)$ by $1+\ln(|\phi^Y_1(w)|\vee 1)$.
Finally, exponentiate and use~\eqref{eq:phi_y} to establish that $\Psi_L^+(w)/\Psi_R^+(w)=o(w)$. 

For $w$ with $|w|>1,\Re(w)\geq 0$ we mimic the steps in~\eqref{eq:lemma} to show that
\begin{align*}&\int_0^1 \frac{1}{t}\e\left|1-e^{-wX_R(t)};X_R(t)>0\right|\D t\\
 &\leq \int_0^1\frac{2|w|\e X^+_R(t/|w|)}{t}\D t+\int_1^{|w|}\frac{2}{t}\D t\leq 2\e Y_1(1)+2\ln|w|.
\end{align*}
A similar bound (with the constant $2r_2\e Y_1(1)$) can be obtained for the term involving $X_L(t)$.
Assuming that $\e Y_1(1)<\infty$ we have $|A(w)|\leq C+4\ln|w|$ for $|w|>1$ and some constant $C$, which yields the result.
If $\e Y_1(1)=\infty$ then we can easily reduce our problem to the one, where the large jumps of $Y_1(t)$ are removed and hence $\e Y_1(1)<\infty$.
\end{proof}

\begin{remark}
 Our proof does not imply that $\Psi_L^+(w)/\Psi_R^+(w)$ is bounded (unless both $X_i(t)$ have bounded variation), 
and hence it leaves the possibility that $F_2(w)\rightarrow 0$ as $w\rightarrow 0$,
which is equivalent to \[\int_0^1 \1{W_1(t)=0}\D L_2(t)=0\] $\p^*$-a.s.\ (and similarly for the reversed indexing).
\end{remark}

\section{Special cases}
\label{sec:special}

\subsection{Deterministic drift in one queue}\label{sec:detdrift}
Suppose $X_2(t)=-d_2 t$ with \mbox{$d_2\geq 0$}, so that $\phi_2(\a)=d_2\a$. Then $X_2(t)-r_2L_1(t)$ is a non-increasing process, and so $W_2(t)= 0$ implying $L_2(t)=r_2L_1(t)+d_2t$, and therefore
\[W_1(t)=X_1(t)-r_1 d_2 t +(1-r_1r_2)L_1(t).\]
But then $W_1(t)$ is a one-dimensional reflection of $X_1(t)-r_1 d_2 t$, and so the generalized Pollaczek-Khinchine formula gives
\begin{equation}\label{eq:example}
 \e e^{-\a_1 W_1-\a_2 W_2}=(d_1+r_1d_2)\frac{\a_1}{\phi_1(\a_1)+r_1d_2\a_1}.
\end{equation}
Let us check if this formula coincides with the result of Theorem~\ref{thm:main}.
Note that $Y_2(t)=0$ and hence $\Psi^-_L(\a)=\Psi^-_R(\a)=1$.
Also $\Psi^+_L(\a)=\frac{p_L}{p_L-r_2\phi_1^Y(\a)}$ and $\Psi^+_R(\a)=\frac{p_R}{p_R-\phi_1^Y(\a)}$.
Therefore, we get
\begin{align*}&\e e^{-\a_1 W_1-\a_2 W_2}=\frac{1}{(1-r_1r_2)(\phi_1(\a_1)+d_2\a_2)}\\
 &\times\left(p_R^0(\a_1-r_2\a_2)+p_L^0(\a_2-r_1\a_1)\frac{p_R(p_L-r_2(d_1^+-\phi_1(\a_1)/\a_1))}{(p_R-(d_1^+-\phi_1(\a_1)/\a_1))p_L}\right),
\end{align*}
which indeed reduces to~\eqref{eq:example}. Let us only check the case when $d_1<0$.
We have $d_1^+=0,p_R=r_1d_2,p_L=d_2,p_R^0=r_1d_2+d_1,p_L^0=d_2+d_1/r_1=p_R^0/r_1$, and so the transform reduces to
\begin{align*}\frac{r_1d_2+d_1}{(1-r_1r_2)(\phi_1(\a_1)+d_2\a_2)}\left(\a_1-r_2\a_2+(\a_2-r_1\a_1)\frac{(d_2+r_2\phi_1(\a_1)/\a_1))}{(r_1d_2+\phi_1(\a_1)/\a_1))}\right),\
\end{align*}
which immediately yields~\eqref{eq:example}.
\subsection{Queues fed by Brownian motion}
Suppose $\phi_i(\a)=d_i\a+\a^2/2$, i.e.\ the driving processes are standard Brownian motions with drifts. Then $\Phi_i(\a)=-d_i+\sqrt{d_i^2+2\a}$ and so
\[\phi^Y_i(\a)=d_i^++\frac{\a}{d_i-\sqrt{d_i^2+2\a}}=\frac{|d_i|-\sqrt{d_i^2+2\a}}{2},\] which corresponds to an inverse Gaussian subordinator.
Hence the processes $X^L$ and $X^R$ can be seen as differences of two inverse Gaussian processes.
Their respective Laplace exponents are given by
\begin{align*}
&\phi_L(w)=\phi^Y_2(-w)+r_2\phi^Y_1(w)=\frac{1}{2}\left(|d_1|r_2+|d_2|-\sqrt{d_2^2-2w}-r_2\sqrt{d_1^2+2w}\right),\\
&\phi_R(w)=\frac{1}{2}\left(|d_1|+r_1|d_2|-r_1\sqrt{d_2^2-2w}-\sqrt{d_1^2+2w}\right).
 \end{align*}
The final step according to Theorem~\ref{thm:main} is to identify the Wiener-Hopf factors corresponding to these Laplace exponents.

\subsection{Compound Poisson input}
This subsection briefly examines relation between the general result given in Theorem~\ref{thm:main} and the result of~\cite{cohen_boxma} for CPP inputs.
Assume that customers arrive into queue $i$ with intensity $\lambda_i$ and bring iid amount of work distributed as $B_i$, and the server speed is $s_i$.
Suppressing the index $i$, the Laplace exponent of the driving process $X(t)$ is \[\phi(\a)=\a s-\lambda+\lambda \e e^{-\a B}.\]
Therefore,
\[\frac{\phi(\a)}{\a}=s-\lambda\e B\frac{1-\e e^{-\a B}}{\a \e B}=s-\rho\e e^{-\a R},\]
where $\rho=\lambda\e B$ and $R$ has the stationary residual life distribution associated to $B$.
This further leads to 
\[\frac{\a}{\Phi(\a)}=s-\rho\e e^{-\a \tau^-_{R}},\]
where $R$ is assumed to be independent of the driving process $X(t)$ and hence of its first passage time $\tau_x^-$. 
Note that $\tau^-_R$ has the interpretation of the length of the busy period in a queue driven by $X(t)$ and started with workload~$R$. 
For simplicity we assume that $s-\rho> 0$ and hence $\tau^-_R$ is a proper positive random variable, which we denote by $U$.

Consider 
\begin{equation}\label{eq:cb}
 -\frac{w}{\Phi_2(-w)}+r_2\frac{w}{\Phi_1(w)}=s_2-\rho_2\e e^{w U_2}+r_2s_1-r_2\rho_1\e ^{-w U_1}
\end{equation}
appearing in~\eqref{eq:tobe_factorized}. Note that this expression can be rewritten as $(s_2+r_2s_1)(1-z \e e^{-w \tilde U})$, where 
$z=(\rho_2+r_2\rho_1)/(s_2+r_2s_1)\in (0,1)$ and $\tilde U$ is a mixture of $U_1$ and $-U_2$. 
This allows to apply Wiener-Hopf factorization for the random walk induced by $\tilde U$ to decompose~\eqref{eq:cb} into a product of functions, 
which are analytic in different half planes, see~\cite{cohen_boxma} for details.

Finally, we mention that a similar technique can be used, when both $X_i(t)$ are \levy processes of bounded variation (a subordinator minus linear drift).
In this case $R$ can be interpreted as an asymptotic overshoot of the corresponding subordinator. This technique fails to generalize further.
In general one can use Wiener-Hopf factorization for \levy processes as is demonstrated in the present paper, which furthermore provides a uniform and neat solution, see Theorem~\ref{thm:main}.

\section*{Acknowledgments}

\bibliography{two_dim.bib}
\end{document}